\documentclass{amsart}

\usepackage[latin1]{inputenc}
\usepackage[T1]{fontenc}
\usepackage{enumerate}

\newtheorem{theorem}{Theorem}[section]
\newtheorem{lemma}[theorem]{Lemma}
\newtheorem{cor}[theorem]{Corollary}
\newtheorem{prop}[theorem]{Proposition}

\theoremstyle{definition}
\newtheorem{definition}[theorem]{Definition}

\theoremstyle{remark}
\newtheorem{remark}[theorem]{Remark}

\numberwithin{equation}{section}

\begin{document}

\title{Hypercyclic subspaces and weighted shifts}

\author[Q. Menet]{Quentin Menet}
\address{Institut de Mathématique\\
Université de Mons\\
20 Place du Parc\\
7000 Mons, Belgique}
\email{Quentin.Menet@umons.ac.be}
\thanks{The author is supported by a grant of FRIA}


\date{}

\begin{abstract}
We first generalize the results of Le\'on and Müller [Studia Math. 175(1) 2006] on hypercyclic subspaces to sequences of operators on Fréchet spaces with a continuous norm. Then we study the particular case of iterates of an operator $T$ and show a simple criterion for having no hypercyclic subspace. Finally we deduce from this criterion a characterization of weighted shifts with hypercyclic subspaces on the spaces $l^p$ or $c_0$, on the space of entire functions and on certain Köthe sequence spaces. We also prove that if $P$ is a non-constant polynomial and $D$ is the differentiation operator on the space of entire functions then $P(D)$ possesses a hypercyclic subspace.
\end{abstract}
\maketitle

\section*{Introduction}

A sequence of continuous linear operators $(T_n)_{n\ge 1}$ between Fréchet spaces $X$, $Y$ is hypercyclic if there exists a vector $x\in X$ (called hypercyclic) such that the set $\{T_nx:n\ge 1\}$ is dense in $Y$. We say that an operator $T:X\rightarrow X$ is hypercyclic if the sequence $(T^n)_{n\ge 1}$ is hypercyclic. An important question about hypercyclic sequences is to know if there exists an infinite-dimensional closed subspace in which every non-zero vector is hypercyclic. Such a subspace is called a hypercyclic subspace. The notion of hypercyclic subspaces is interesting because some hypercyclic operators, like the translation operators on the space of entire functions, possess a hypercyclic subspace (see \cite{Bernal}) while some others, like scalar multiples of the backward shift on $l^p$, do not possess any hypercyclic subspace (see \cite{Montes}).

In the case of a weakly mixing operator on a separable complex Banach space, Gonz\'alez, Le\'on and Montes have obtained the following characterization:
\begin{theorem}[\cite{Gonzalez}]\label{Montes}
Let $T$ be a weakly mixing operator on a separable complex Banach space $X$. Then the following assertions are equivalent:
\begin{enumerate}[\upshape (1)]
\item $T$ has a hypercyclic subspace;
\item there exist an increasing sequence $(n_k)$ of positive integers and an infinite-dimensional closed subspace $M_0$ of $X$ such that $T^{n_k}x\rightarrow 0$ for all $x\in M_0$;
\item there exist an increasing sequence $(n_k)$ of positive integers and an infinite-dimensional closed subspace $M_b$ of $X$ such that $\sup_k\|T^{n_k}|_{M_b}\|<\infty$;
\item the essential spectrum of $T$ intersects the closed unit disk.
\end{enumerate}
\end{theorem}

The implication from (2) to (1) was already proved in 1996 if $T$ satisfies the Hypercyclicity Criterion for $(n_k)$ (see \cite{Montes}) and was even generalized to operators on separable Fréchet spaces with a continuous norm (see \cite{Bonet}, \cite{Petersson}) and to sequences of operators on separable Banach spaces satisfying a certain condition $(C)$ (see \cite{Leon}).
\begin{definition}[\cite{Leon}]\label{def C ban}
A sequence $(T_n)\subset L(X,Y)$, the space of continuous linear operators from $X$ to $Y$, with $X$, $Y$ Banach spaces, satisfies \emph{condition} (C) if there exist an increasing sequence $(n_k)$ of positive integers and a dense subset $X_0\subset X$ such that
\begin{enumerate}
\item $T_{n_k}x\rightarrow 0$ for every $x\in X_0$;
\item $\bigcup_k T_{n_k}(\{x\in X:\|x\|<1\})$ is dense in $Y$.
\end{enumerate}
\end{definition}
In their article \cite{Leon}, Le\'on and Müller have also highlighted the following two criteria:
\begin{theorem}[{\cite[Theorem 20]{Leon}}]\label{subspace ban}
Let $(T_n)\subset L(X,Y)$ be a sequence of operators between separable Banach spaces $X$, $Y$ satisfying condition $(C)$ for a sequence $(n_k)$. Suppose that there are infinite-dimensional closed subspaces $M_1,M_2,\ldots$ such that $X\supset M_1\supset M_2 \supset \dots$ and $\sup_k\|T_{n_k}|_{M_k}\|<\infty$. Then $(T_n)$ possesses a hypercyclic subspace.
\end{theorem}
\begin{theorem}[{\cite[Corollary 23]{Leon}}]\label{codim ban}
Let $(T_n)\subset L(X,Y)$ be a sequence of operators between separable Banach spaces $X$, $Y$. Suppose that there are closed subspaces $E_n\subset X$ $(n\ge1)$ of finite codimension and numbers $C_n$ $(n\ge 1)$ with $C_n\rightarrow \infty$ such that
\[\|T_nx\|\ge C_n \|x\| \quad \text{for any }x\in E_n, n\ge 1.\]
Then $(T_n)$ does not possess any hypercyclic subspace.
\end{theorem}
In the case of an operator $T$ on a separable Banach space, condition $(C)$ is equivalent to being weakly mixing (see \cite{Leon}) and it follows from the proof of Theorem \ref{Montes} that the above two criteria are also necessary conditions if $T$ is an operator on a separable \emph{complex} Banach space. In their book \cite{Karl}, Grosse-Erdmann and Peris have proved that Theorem \ref{codim ban} can be stated for an operator on a Fréchet space.

We prove in the first section that the implication from (2) to (1) in Theorem \ref{Montes}, and Theorems \ref{subspace ban} and \ref{codim ban} are still true for sequences of operators between Fréchet spaces with a continuous norm.

In the second section, we obtain some conditions that are equivalent to the criterion of Theorem \ref{codim ban} in the case of an operator $T$ on a Banach space or on a Fréchet space. If $T$ is an operator on a Banach space, this result is the following:
\begin{theorem}\label{thmbanach}
Let $X$ be a Banach space and $T:X\rightarrow X$ a continuous linear operator. The following conditions are equivalent:
\begin{enumerate}[\upshape (i)]
\item there exist a sequence of numbers $(C_n)$ with $C_n\rightarrow \infty$ and a sequence of subspaces $(E_n)$ of finite codimension such that for every $n\ge 1$,
\[\|T^nx\|\ge C_n \|x\| \quad \text{for any } x\in E_n;\]
\item there exist $C>1$, a subspace $E$ of finite codimension and an integer $n\ge 1$ such that
\[\|T^{n}x\|\ge C \|x\| \quad \text{for any } x\in E;\]
\item there exist $C>1$, a subspace $E$ of finite codimension and an integer $n\ge 1$ such that for any $x\in E$, there exists an integer $k\le n$ for which
\[\|T^kx\|\ge C \|x\|.\]
\end{enumerate}
\end{theorem}
Therefore, if $T$ is a weakly mixing operator on a separable \emph{complex} Banach space, $T$ does not possess any hypercyclic subspace if and only if $T$ satisfies one of these conditions. Condition (ii) is easier to verify than condition (i) and thus particularly useful for applications.

In the third and the fourth sections, we are interested in weighted shifts. In the article \cite{LeonMontes}, a characterization of weighted shifts on the complex space $l^2$ with hypercyclic subspaces is given using spectral theory. Using the ideas of Theorem \ref{thmbanach}, we show how to obtain a characterization of the weighted shifts with a hypercyclic subspace on the (complex or real) spaces $l^p$ and $c_0$ but also on the space of entire functions and on certain Köthe sequence spaces. In 2010, Shkarin has answered a question that had long remained open. He has shown that the differentiation operator $D$ on the space of entire functions possesses a hypercyclic subspace (see \cite{Shkarin}). As this operator can be seen as a weighted shift, our result improves that of Shkarin by determining which weighted shifts on the space of entire functions possess a hypercyclic subspace. In fact, that seems be the first time that we can determine which operators possess a subspace hypercyclic for a class of operators on a Fréchet space.

To finish, we consider functions of weighted shifts in the fifth section. In particular, we show that if $P$ is a non-constant polynomial then $P(D)$ possesses a hypercyclic subspace. This result completes Petersson's result that states that if $\phi$ is an entire function of exponential type that is not a polynomial then $\phi(D)$ has a hypercyclic subspace (see \cite{Petersson}).

\section[Hypercyclic subspaces on a Fréchet space]{Hypercyclic subspaces for a sequence of operators on a Fréchet space with a continuous norm}

The results on hypercyclic subspaces of Le\'on and Müller in their article \cite{Leon} are stated for sequences of operators between Banach spaces satisfying a certain condition $(C)$ (see Definition \ref{def C ban}). This condition can be generalized to sequences of operators on Fréchet spaces: 

\begin{definition}\label{def C fre}
Let $X$ be a Fréchet space and $Y$ a topological vector space. A sequence $(T_n)\subset L(X,Y)$ satisfies \emph{condition} (C) if there exist an increasing sequence $(n_k)$ of positive integers and a dense subset $X_0\subset X$ such that
\begin{enumerate}
\item $T_{n_k}x\rightarrow 0$ for every $x\in X_0$;
\item for every continuous seminorm $p$ on $X$, $\bigcup_k T_{n_k}(\{x\in X: p(x)<1\})$ is dense in $Y$.
\end{enumerate}
\end{definition}
\begin{remark}
If $(p_n)$ is an increasing sequence of seminorms defining the topology of $X$ then the second condition is equivalent to having that for every $n$, $\bigcup_k T_{n_k}(\{x\in X: p_n(x)<1\})$ is dense in $Y$.
\end{remark}
\begin{remark}
In the case of a sequence of operators between Banach spaces, Definition \ref{def C fre} is obviously identical to Definition \ref{def C ban}. Moreover, if $T$ is an operator on a separable Fréchet space, it is evident that if $T$ satisfies the Hypercyclicity Criterion for $(n_k)$ then $T$ satisfies condition $(C)$ for $(n_k)$. In fact, one can show by adapting the ideas in \cite{Leon} to Fréchet spaces that $T$ satisfies condition $(C)$ if and only if $T$ satisfies the Hypercyclicity Criterion.
\end{remark}

To generalize the results on hypercyclic subspaces of Le\'on and Müller to sequences of operators on Fréchet spaces with a continuous norm, we need some results on basic sequences in Fréchet spaces with a continuous norm.

\begin{definition}
A sequence $(e_k)_{k\ge 1}$ in a Fréchet space is called \emph{basic} if for every $x\in \overline{\text{span}}\{e_k:k\ge 1\}$, there exists a unique sequence $(a_k)_{k\ge1}$ in $\mathbb{K}$ ($\mathbb{K}=\mathbb{R}$ or $\mathbb{C}$) such that $x=\sum_{k=1}^{\infty}a_ke_k$.
\end{definition}

Let $X$ be a Fréchet space with a continuous norm. The existence of such a norm implies the existence of an increasing sequence of norms $(p_n)$ defining the topology of $X$. Using the fact that $X_n:=(X,p_n)$ is a normed space, we can deduce some results for basic sequences on $X$ from results for basic sequences in a Banach space. For example, if a sequence $(e_k)_{k\ge 1}\subset X$ is basic in $X_n$ for any $n\ge 1$ then this sequence is also basic in $X$. By a basic sequence in $X_n$ we mean a sequence that is basic in the completion of $X_n$. This technique has been used by Petersson in his article \cite{Petersson} and by the author in \cite{Menet}.

\begin{lemma}[\cite{Menet}]\label{lem basic}
Let $X$ be a Fréchet space with a continuous norm, $(p_n)$ an increasing sequence of norms defining the topology of $X$ and $(\varepsilon_n)_{n\ge 1}$ a sequence of positive numbers with $\prod_{n} (1+\varepsilon_n)=K<\infty$.
If $(e_k)_{k\ge 1}$ is a sequence of nonzero vectors in $X$ such that for any $n\ge 1$, for any $j\le n$, for any $a_1,\ldots,a_{n+1}\in \mathbb{K}$,
\[p_j\Big(\sum_{k=1}^n a_k e_k\Big)\le (1+\varepsilon_n)p_j\Big(\sum_{k=1}^{n+1}a_k e_k\Big)\]
then this sequence is basic in $X_n$ for any $n\ge 1$ and thus in $X$; moreover for any $n\ge 1$, the sequence $(e_k)_{k\ge n}$ is basic in $X_n$ with basic constant less than $K$.
\end{lemma}

\begin{lemma}[{\cite[Lemma 10.39]{Karl}}]\label{lem karl}
Let $X$ be a Fréchet space, $F$ a finite-dimensional subspace of $X$, $p$ a continuous seminorm on $X$ and $\varepsilon>0$. Then there exists a closed subspace $L$ of finite codimension such that for any $x\in L$ and $y\in F$,
\[p(x+y)\ge \max\Big(\frac{p(x)}{2+\varepsilon},\frac{p(y)}{1+\varepsilon}\Big).\]
\end{lemma}

With these two lemmas we can easily construct recursively a basic sequence $(e_n)_{n\ge 1}$ in an infinite-dimensional Fréchet space with a continuous norm and we can even choose each vector $e_n$ in an infinite-dimensional subspace.

\begin{lemma}\label{exist basic}
Let $X$ be an infinite-dimensional Fréchet space with a continuous norm, $(p_n)$ an increasing sequence of norms defining the topology of $X$, $(M_n)$ a sequence of infinite-dimensional subspaces and a number $K>1$. Then there exists a basic sequence $(e_n)_{n\ge 1}$ in $X$ such that for any $n\ge1$ we have $p_1(e_n)=1$, $e_n\in M_n$ and the sequence $(e_k)_{k\ge n}$ is basic in $X_n$ with basic constant less than $K$. 
\end{lemma}

An important notion for basic sequences is the equivalence between two basic sequences.

\begin{definition}
Let $X$ be a Fréchet space.
Two basic sequences $(e_n)_n$ and $(f_n)_n$ in $X$ are called \emph{equivalent} if for every sequence $(a_n)_{n\ge 1}$ in $\mathbb{K}$, the series $\sum_{n=1}^{\infty} a_n e_n$ converges in $X$ if and only if $\sum_{n=1}^{\infty} a_n f_n$ converges in $X$.
\end{definition}

Theorem V.9 in \cite{Diestel} for Banach spaces implies the following lemma for Fréchet spaces with a continuous norm:

\begin{lemma}[\cite{Menet}]\label{lem equiv}
Let $X$ be a Fréchet space with a continuous norm and $(p_n)_n$ an increasing sequence of norms defining the topology of $X$. Suppose that $(e_n)_{n\ge 1}$ is a basic sequence in $X$ such that for any $n\ge 1$, we have $p_1(e_n)=1$ and the sequence $(e_k)_{k\ge n}$ is basic in $X_n$ with basic constant less than $K$. If $(f_k)_{k\ge 1}\subset X$ is a sequence satisfying \[\sum_{n=1}^{\infty}2Kp_n(e_n-f_n)<1\]
then $(f_k)_{k\ge 1}$ is a basic sequence in $X$ and in $X_n$ for any $n\ge1$.
Moreover the sequences $(e_k)_{k\ge 1}$ and $(f_k)_{k\ge 1}$ are equivalent in $X$ and in $X_n$ for any $n\ge1$.
\end{lemma}

Now we can extend the results on hypercyclic subspaces in the article \cite{Leon} to Fréchet spaces with a continuous norm.
For the rest of this section we suppose that $X$ is an infinite-dimensional Fréchet space with a continuous norm, $Y$ is a separable Fréchet space and $(T_n)$ is a sequence of continuous linear operators from $X$ to $Y$.

\begin{theorem}\label{thm M0}
If $(T_n)$ satisfies condition $(C)$ for a sequence $(n_k)$ and if there exists an infinite-dimensional closed space $M_0$ of $X$ such that $T_{n_k}x\rightarrow 0$ for all $x\in M_0$ then $(T_n)$ possesses a hypercyclic subspace.
\end{theorem}
\begin{proof}
We consider an increasing sequence $(p_n)$ of norms defining the topology of $X$ and an increasing sequence $(q_n)$ of seminorms defining the topology of $Y$. By Lemma \ref{exist basic}, there exists a basic sequence $(e_n)_{n\ge 1}$ in $M_0$ such that for every $n\ge 1$, we have $p_1(e_n)=1$ and the sequence $(e_k)_{k\ge n}$ is basic in $X_n$ with basic constant less than $2$. Let $(y_k)_{k\ge 1}$ be a dense sequence in $Y$ and $\prec$ the order on $\mathbb{N}\times\mathbb{N}$ defined by $(i,j)\prec(i',j')$ if $i+j<i'+j'$ or if $i+j=i'+j'$ and $i<i'$.

We construct a family $(z_{i,j})_{i,j\ge 1}\subset X_0$ where $X_0$ is the dense set given by condition $(C)$ and a family $(n_{i,j})_{i,j\ge 1}\subset \mathbb{N}$ such that for any $i\ge 1$, $(n_{i,j})_{j\ge 1}$ is a subsequence of $(n_k)$. If $z_{i',j'}$ and $n_{i',j'}$ are already constructed for every $(i',j')\prec(i,j)$ then we choose $z_{i,j}\in X_0$ and $n_{i,j}\in (n_k)$ with $n_{i,j}>\max\{n_{i',j'}:(i',j')\prec(i,j)\}$ such that for any $(i',j')\prec(i,j)$,
\begin{gather}
q_{i+j}(T_{n_{i,j}}z_{i',j'})<\frac{1}{2^{i'+j'+j}}, \label{n grand}\\
q_{i+j}(T_{n_{i',j'}}z_{i,j})<\frac{1}{2^{i+j+j'}},  \label{T z petit}\\
q_{i+j}(T_{n_{i,j}}z_{i,j}-y_j)<\frac{1}{2^{j}} \text{ and } p_{i+j}(z_{i,j})<\frac{1}{2^{i+j+2}}\label{z petit}.
\end{gather}
Satisfying all these assertions is possible because if for every continuous seminorm $p$ on $X$, $\bigcup_k T_{n_k}(\{x\in X: p(x)<1\})$ is dense in $Y$ then for every continuous seminorm $p$ on $X$, for every $\varepsilon>0$, for every $N\ge 1$, the set $\bigcup_{k\ge N} T_{n_k}(\{x\in X: p(x)<\varepsilon\})$ is also dense in $Y$.

We define, for any $i\ge 1$,
\[z_i:=e_i+\sum_{j=1}^{\infty} z_{i,j}.\]
By \eqref{z petit}, these series are convergent and using Lemma \ref{lem equiv}, we deduce that the sequence $(z_i)_{i\ge 1}$ is a basic sequence equivalent to $(e_i)_{i\ge 1}$ in $X$. 
Let $M$ be the closed linear span of $(z_i)$ and $z\in M\backslash\{0\}$. We need to show that $z$ is hypercyclic for the sequence $(T_n)$. Since $(z_i)$ is a basic sequence, we know that $z=\sum_{i=1}^{\infty}a_i z_i$ and we can suppose that $a_{k}=1$ for some $k$. By equivalence between the basic sequences $(z_i)_i$  and $(e_i)_i$, we deduce that $\sum_{i=1}^{\infty}a_i e_i$ also converges. There thus exists $K>0$ such that for any $i\ge 1$, we have $|a_i|\le K$. Indeed, as the sequence $(e_i)_{i\ge 1}$ is basic in $X_1$ with basic constant less than $2$ and $p_1(e_n)=1$, we have
\begin{equation}
|a_i|=p_1(a_i e_i)\le p_1\Big(\sum_{l=1}^{i}a_l e_l\Big)+ p_1\Big(\sum_{l=1}^{i-1}a_l e_l\Big)\le 4p_1\Big(\sum_{l=1}^{\infty}a_l e_l\Big).
\label{ak maj}
\end{equation}
Let $n\ge 1$ and $r\ge n$. We have, using \eqref{n grand}, \eqref{T z petit} and \eqref{z petit},
\begin{align*}
q_n(T_{n_{k,r}}z-y_r)&\le \sum_{(i,j)\prec(k,r)} |a_i|q_n(T_{n_{k,r}}z_{i,j}) + \sum_{(i,j)\succ(k,r)} |a_i|q_n(T_{n_{k,r}}z_{i,j})\\
&\quad + q_n(T_{n_{k,r}}z_{k,r}-y_r) + q_n\Big(T_{n_{k,r}}\Big(\sum_{i\ge1} a_ie_i\Big)\Big)\\
&\le\sum_{(i,j)\prec(k,r)} K q_{k+r}(T_{n_{k,r}}z_{i,j})+\sum_{(i,j)\succ(k,r)} K q_{i+j}(T_{n_{k,r}} z_{i,j})\\
&\quad+ q_n(T_{n_{k,r}}z_{k,r}-y_r) + q_n\Big(T_{n_{k,r}}\Big(\sum_{i\ge 1} a_ie_i\Big)\Big)\\
&\le \sum_{(i,j)\ne(k,r)} K \frac{1}{2^{i+j+r}} + \frac{1}{2^r} +q_n\Big(T_{n_{k,r}}\Big(\sum_{i\ge 1} a_ie_i\Big)\Big)\\
&\le \frac{K+1}{2^{r}}+ q_n\Big(T_{n_{k,r}}\Big(\sum_{i\ge 1} a_ie_i\Big)\Big) \underset{r\rightarrow \infty}{\longrightarrow} 0 \quad \text{because $\sum_{i\ge 1} a_ie_i\in M_0$.}
\end{align*}
\end{proof}

\begin{theorem}\label{thm M}
Let $(p_n)$, $(q_n)$ be sequences of seminorms defining the topology of $X$, $Y$.
If $(T_n)$ satisfies condition $(C)$ for a sequence $(n_k)$ and if there exists a non-increasing sequence of infinite-dimensional closed spaces $(M_j)_{j\ge 1}$ of $X$ such that for every $n\ge 1$, there exist a positive number $C_n$ and two integers $m(n),k(n)\ge 1$ such that we have  for any $j\ge k(n)$, for any $x\in M_j$,
\[q_n(T_{n_j}x)\le C_n p_{m(n)}(x)\]
then $(T_n)$ possesses a hypercyclic subspace.
\end{theorem}
\begin{proof}
Without loss of generality, we can suppose that $(p_n)$ is an increasing sequence of norms defining the topology of $X$ and $(q_n)$ is an increasing sequence of seminorms defining the topology of $Y$. We remark that, by using the previous theorem, it is sufficient to construct an infinite-dimentional closed subspace $M_0$ of $X$ and a subsequence $(n_{k_l})$ of $(n_k)$ such that for any $x\in M_0$, we have $\lim T_{n_{k_l}} x=0$ and $(T_n)$ satisfies condition $(C)$ for the sequence $(n_{k_l})$. We thus consider a basic sequence $(e_j)_{j\ge1}$ in $X$ such that we have $p_1(e_j)=1$, $e_j\in M_j$ and for every $n\ge 1$, the sequence $(e_j)_{j\ge n}$ is basic in $X_n:=(X,p_n)$ with basic constant less than $2$. The existence of the sequence $(e_j)$ is guaranteed by Lemma \ref{exist basic}.

Let $(y_k)_{k\ge 1}$ be a dense sequence in $Y$ and $X_0$ the dense set given by condition $(C)$ for the sequence $(n_k)$. By continuity of the operators $T_n$, we can find a sequence $(f_j)\subset X_0$ such that 
for any $j$, for any $k\le j$,
\begin{equation}
q_j(T_{n_k}e_j-T_{n_k}f_j)<\frac{1}{2^{j}} \text{\ and \ } p_j(e_j-f_j)<\frac{1}{2^{j+2}}.
\label{prop f}
\end{equation}
If we denote $d_{q_l}(A,x)=\inf_{a\in A}q_l(a-x)$ then since the sequence $(f_j)$ is included in $X_0$, we can also find an increasing sequence of integers $(k_l)_{l\ge 1}$ such that for any $l\ge 1$, for any $n,j\le k_{l-1}$,
\begin{equation}
q_n(T_{n_{k_l}}f_j)<\frac{1}{2^{l+j}} \ \text{and} \ d_{q_l}(T_{n_{k_l}}(\{x\in X: p_l(x)<1\}),y_l)<\frac{1}{2^l}
\label{prop k}
\end{equation}
where we suppose that $k_0=1$.
Using Lemma \ref{lem equiv} and the second inequality in \eqref{prop f}, we deduce that $(f_{k_l})_{l\ge 1}$ is a basic sequence equivalent to $(e_{k_l})_{l\ge 1}$ in $X$.

Let $M_0$ be the closed linear span of $(f_{k_l})_{l\ge 1}$ in $X$ and $x$ a vector in $M_0$. We know that we have $x=\sum_{k=1}^{\infty} a_j f_{k_j}$ where the sequence $(a_j)$ is bounded by some constant $K$ (see \eqref{ak maj}). Let $n\ge 1$ and $l\ge 2$ with $n\le k_{l-1}$ and $k(n)\le k_l$. Since $\sum_{j=l}^{\infty}a_j e_{k_j}\in M_{k_l}$, we deduce from \eqref{prop f} and \eqref{prop k} that
\begin{align*}
q_n(T_{n_{k_l}}x)&\le \sum_{j=1}^{l-1}|a_j| q_n(T_{n_{k_l}}f_{k_j})
+ \sum_{j=l}^{\infty} |a_j| q_n(T_{n_{k_l}}(f_{k_j}-e_{k_j}))
+ q_n\Big(T_{n_{k_l}}\Big(\sum_{j=l}^{\infty}a_je_{k_j}\Big)\Big)\\
&\le \sum_{j=1}^{l-1} \frac{K}{2^{l+k_j}}
+\sum_{j=l}^{\infty} K q_{k_j}(T_{n_{k_l}}(f_{k_j}-e_{k_j}))
+ C_n p_{m(n)}\Big(\sum_{j=l}^{\infty}a_j e_{k_j}\Big)\\
&\le \frac{lK}{2^{l}}
+ \sum_{j=l}^{\infty} \frac{K}{2^{k_j}}
+ C_n p_{m(n)}\Big(\sum_{j=l}^{\infty}a_j e_{k_j}\Big) \underset{l\rightarrow \infty}{\longrightarrow} 0.
\end{align*}
Since for all $n\ge 1$, we have also by \eqref{prop k} that the set $\bigcup_l T_{n_{k_l}}(\{x\in X: p_n(x)<1\})$ is dense in $Y$ so that $(T_n)$ satisfies condition $(C)$ for the sequence $(n_{k_l})$, we deduce the desired result from Theorem \ref{thm M0}.
\end{proof}
\begin{remark}
In Theorem \ref{thm M}, the second condition is equivalent to the existence of a non-increasing sequence of infinite-dimensional closed spaces $(M_j)_{j\ge 1}$ of $X$ such that for every continuous seminorm $q$ on $Y$, there exists a positive number $C$, an integer $k\ge 1$ and a continuous seminorm $p$ on $X$ such that we have for any $j\ge k$, for any $x\in M_j$,
\[q(T_{n_j}x)\le C p(x).\]
\end{remark}

In the next theorem, we do not suppose that $X$ possesses a continuous norm but that there exists a continuous seminorm $p$ on $X$ such that $p(x)>0$ for every hypercyclic vector $x$ in $X$. Obviously, if a continuous norm exists then this condition is always satisfied.

\begin{theorem}\label{no subspace}
Let $(p_n)$ be an increasing sequence of seminorms defining the topology of $X$.
If there exists a continuous seminorm $p$ on $X$ such that $p(x)>0$ for every hypercyclic vector $x$ in $X$ and if there exist a sequence of subspaces $(E_n)_{n\ge 1}$ of finite codimension, a sequence $(C_n)_{n\ge 1}$ of numbers with $C_n\rightarrow \infty$ and a continuous seminorm $q$ on $Y$ such that for any $n\ge 1$, for any $x\in E_n$, we have
\[q(T_nx)\ge C_n p_n(x)\]
then $(T_n)$ does not possess any hypercyclic subspace.
\end{theorem}
\begin{proof}
If $p$ is a continuous seminorm on $X$, there exists an integer $N\ge 1$ and a positive number $K>0$ such that for every $x\in X$, we have $p(x)\le K p_N(x)$. Without loss of generality, we can thus suppose that we have $p_1(x)>0$ for every hypercyclic vector $x$ in $X$.
Assume that $M$ is a hypercyclic subspace; we show that there exists a vector $x\in M$ such that  \[\lim_{j\rightarrow\infty}q(T_j(x))=\infty\]
which is a contradiction.
Since $C_n$ tends to infinity, there exists an increasing sequence $(k_n)_{n\ge 1}$ of positive integers such that for any $n\ge 2$, for any $k_{n-1}< j\le k_n$, we have 
\begin{equation}
n^3\le C_j.\label{nC}
\end{equation}
We then construct a sequence $(e_n)_{n\ge 1}\subset M$ recursively such that
\begin{enumerate}
\item $p_n(e_n)=\frac{1}{n^2}$;
\item $e_n\in \bigcap_{k_{n-1}< j\le k_n} E_{j}$;
\item for any $j\le k_n$, $T_je_n\in \bigcap_{k\le n-1} L_{k,j}$, where $L_{k,j}$ is the closed subspace of finite codimension given by Lemma \ref{lem karl} for $F_{k,j}=\text{span}(T_je_1,\dots, T_je_k)$, the seminorm $q$ and $\varepsilon>0$.
\end{enumerate}
Such a construction is possible because $M$ is an infinite-dimensional space and each of the intersections $\bigcap_{k_{n-1}< j\le k_n} E_{j}$ and $\bigcap_{k\le n-1} L_{k,j}$ are spaces of finite codimension.
Moreover we can choose $e_n$ such that $p_n(e_n)=\frac{1}{n^2}$ because we have $p_1(x)>0$ for every hypercyclic vector $x$ and $M$ is supposed to be a hypercyclic subspace. 
Letting $x:=\sum_{\nu=1}^{\infty} e_\nu$, we have $x\in M$ and if $k_{n-1}<j\le k_n$,
we then have
\begin{align*}
q(T_jx)&=q\Big(\sum_{\nu=1}^{\infty}T_je_{\nu}\Big)
\ge \frac{1}{1+\varepsilon}q\Big(\sum_{\nu=1}^{n}T_je_{\nu}\Big) \quad \text{because }\sum_{\nu=n+1}^{\infty}T_je_{\nu}\in L_{n,j}\\
&\ge \frac{1}{(1+\varepsilon)(2+\varepsilon)}q(T_je_n) \quad \text{because }T_je_n\in L_{n-1,j}\\
&\ge \frac{C_j p_j(e_n)}{(1+\varepsilon)(2+\varepsilon)}
\ge \frac{C_j p_n(e_n)}{(1+\varepsilon)(2+\varepsilon)} \quad \text{because }e_n\in E_j \text{\ and\ } n-1\le k_{n-1}< j\\
&\ge \frac{n}{(1+\varepsilon)(2+\varepsilon)}\quad \text{by \eqref{nC}}.
\end{align*}
The result follows.
\end{proof}

\section{Criterion for having no hypercyclic subspace}

The criterion of Theorem \ref{no subspace} can be significantly simplified in the case of an operator $T:X\rightarrow X$. By operator we mean a continuous linear operator. For convenience, we begin by the case where $X$ is a Banach space.

\begin{proof}[Proof of Theorem \ref{thmbanach}]
The implications (i)$\Rightarrow$(ii) and (ii)$\Rightarrow$(iii) are evident.

(iii)$\Rightarrow$(i).
Let $C>1$, $E$ a subspace of finite codimension and $m\ge 1$ an integer such that for any $x\in E$, there exists an integer $k\le m$ for which
\[\|T^kx\|\ge C \|x\|.\]
Let $n\ge 1$. We consider $K=\sup_{0\le k< m}\|T^k\|$ and $E_n=\bigcap_{k=0}^{n-1}T^{-k}E$.
Let $x\in E_n$. Since $E_n\subset E$, there exists $0< k_1\le m$ such that we have
\[\|T^{k_1}x\|\ge C \|x\|.\]
If $k_1< n$, we have $T^{k_1}x\in E$ and there thus exists $k_1< k_2\le k_1+m$
such that
\[\|T^{k_2}x\|\ge C \|T^{k_1}x\|\ge C^2 \|x\|.\]
We can repeat this argument as long as we have $k_l<n$. Therefore, if $n=jm+p$ with $j\ge 0$ and $0\le p<m$ then we are sure to can use this argument at least $j$ times. As $C>1$, we deduce that there exists $n\le k < n+m$ such that
\[\|T^{k}x\|\ge C^{j} \|x\|.\]
We thus have for every $x\in E_n$, for some $n\le k < n+m$,
\[ C^{j}\|x\|\le \|T^{k}x\| \le K \|T^{n}x\|.\]
Letting $C_n=\frac{C^{j}}{K}$, we obtain the desired result.
\end{proof}

If we denote $\text{cofin}:=\{E\subset X: E \text{ subspace of finite codimension}\}$ then as condition (i) is equivalent to having no hypercyclic subspace in the case of weakly mixing operators on a separable complex Banach space, we obtain in particular the following new characterization:

\begin{cor}\label{corban}
Let $X$ be a separable complex Banach space and $T:X\rightarrow X$ a weakly mixing operator.
The operator $T$ possesses a hypercyclic subspace if and only if
\[\sup_{n\ge 1}\sup_{E\in \text{cofin}}\inf_{x\in E\backslash \{0\}} \frac{\|T^nx\|}{\|x\|}\le 1\]
or equivalently if and only if 
\[\sup_{n\ge 1}\sup_{E\in \text{cofin}}\inf_{x\in E\backslash \{0\}} \frac{\|T^nx\|}{\|x\|}<\infty.\]
\end{cor}

We generalize Theorem \ref{thmbanach} for operators on a Fréchet space.

\begin{theorem}\label{thmfrechet}
Let $X$ be a Fréchet space and $T:X\rightarrow X$ an operator. 
The following conditions are equivalent:
\begin{enumerate}[\upshape (i)]
\item there exist an increasing sequence of seminorms $(p_n)$ defining the topology of $X$, a sequence of numbers $(C_n)$ with $C_n\rightarrow \infty$, a sequence of subspaces $(E_n)$ of finite codimension and $N\ge 1$ such that for every $n\ge 1$,
\[p_N(T^nx)\ge C_n p_n(x) \quad \text{for any } x\in E_n;\]
\item there exist a sequence of seminorms $(p_n)$ defining the topology of $X$ and $N\ge 1$ such that for every $n\ge 1$, there exist $C_n>1$, a subspace $E_n$ of finite codimension and an integer $m_n\ge 1$ such that
\[p_N(T^{m_n}x)\ge C_n p_n(x) \quad \text{for any } x\in E_n;\]
\item there exist a sequence of seminorms $(p_n)$ defining the topology of $X$ and $N\ge 1$ such that for every $n\ge 1$, there exist $C_n>1$, a subspace $E_n$ of finite codimension and an integer $m_n\ge 1$ such that for any $x\in E_n$, there exists an integer $1\le k\le m_n$ for which
\[p_N(T^kx)\ge C_n p_n(x).\]
\end{enumerate}
\end{theorem}
\begin{proof}
The implications (i)$\Rightarrow$(ii) and (ii)$\Rightarrow$(iii) are evident.

(iii)$\Rightarrow$(i). By hypothesis, there exists $N\ge 1$ such that for every $n\ge 1$, there exist $C_n>1$, a subspace $E_n$ of finite codimension and an integer $m_n\ge 1$ such that for any $x\in E_n$, there exists an integer $1\le k\le m_n$ for which
\[p_N(T^kx)\ge C_n p_n(x).\]
Let $l\ge 1$. For any $n\ge 1$, we consider  $E_{n,l}=\bigcap_{k=0}^{n-1}T^{-k}E_N\cap E_l$.
Let $x$ be a vector in $E_{n,l}$. There exists $k_1\le m_l$ such that 
\[p_N(T^{k_1}x)\ge C_l p_l(x).\]
Moreover if $n=m_l+jm_N+p$ with $j\ge 1$ and $0\le p<m_N$ then $T^{k_1}x\in E_N$ and we deduce as in the proof of Theorem \ref{thmbanach} that there exists an integer $n\le k< n+m_N$ such that
\[p_{N}(T^{k}x)\ge C^j_Np_N(T^{k_1}x)\ge C^j_NC_lp_l(x).\]
If for $j\ge 1$, we denote $q_j=\max(p_1,\dots,p_j)$ then by continuity of $T$, there exist a positive number $K$ and an integer $J\ge 1$ such that for any $0\le m< m_N$, we have
\[p_N(T^mx)\le K q_{J}(x)\quad\text{for any $x\in X$} \]
and thus for any $0\le m< m_N$, for any $k\ge m$, we have
\[q_J(T^{k-m} x)\ge \frac{1}{K}p_{N}(T^{k}x)\quad\text{for any $x\in X$}.\]
If $n=m_l+jm_N+p$ with $j\ge 1$ and $0\le p<m_N$, it follows that for every vector $x\in E_{n,l}$, there exists $n\le k< n+m_N$ such that
\begin{equation}
q_J(T^nx)= 
q_J(T^{k-(k-n)}x)
\ge \frac{1}{K}p_N(T^{k}x)
\ge \frac{C_N^jC_l}{K}p_l(x).\label{add}\end{equation}
To conclude, we have just to construct a convenient increasing sequence of seminorms $(p'_n)$ defining the topology of $X$. We can suppose that $(m_l)_{l\ge 1}$ is an increasing sequence. So if for $n< m_1+m_N$, we let $p'_n=q_1$, $C'_n=0$, $E'_n=X$, and for $m_l+lm_N\le n< m_{l+1}+(l+1)m_N$, we let $p'_n=q_l$, $C'_n=\frac{C_N^{l}}{K}$ and $E'_n=\bigcap_{k\le l}E_{n,k}$, we have the desired inequalities for $(C'_n)$, $(E'_n)$, $(p'_n)$ and $q_J$. Indeed, if $n\ge m_l+lm_N$ and $x\in E'_n$, then for any $k\le l$, we have by \eqref{add},
\[q_J(T^nx)\ge \frac{C^l_NC_k}{K}p_k(x)\ge \frac{C^l_N}{K}p_k(x)\]
and thus if $m_l+lm_N\le n< m_{l+1}+(l+1)m_N$, we get
\[q_J(T^nx)\ge C'_nq_l(x)=C'_np'_n(x).\]
\end{proof}

By Theorem \ref{no subspace}, we know that if $X$ is a Fréchet space with a continuous norm, condition (i) in Theorem \ref{thmfrechet} implies the non-existence of hypercyclic subspaces.

\begin{cor}\label{corfre}
Let $X$ be a Fréchet space with a continuous norm and $T:X\rightarrow X$ an operator.
If there exist a sequence of seminorms $(p_j)$ defining the topology of $X$ and $J\ge 1$ such that for every $j\ge 1$,
\[\sup_{n\ge 1}\sup_{E\in \text{cofin}}\inf_{x\in E\backslash \ker(p_j)} \frac{p_J(T^nx)}{p_j(x)}> 1\]
or equivalently if 
\[\sup_{n\ge 1}\sup_{E\in \text{cofin}}\inf_{x\in E\backslash \ker(p_j)} \frac{p_J(T^nx)}{p_j(x)}=\infty\]
then $T$ does not possess any hypercyclic subspace.
\end{cor}
\begin{remark}
Unlike the case of complex Banach spaces, we do not know if, in general, these criteria are also necessary conditions for having no hypercyclic subspace. So far no characterization is known of operators on Fréchet spaces with a continuous norm.
\end{remark}

\section{Weighted shifts on complex spaces $l^p$ and $c_0$}
For weighted shifts on complex spaces $l^p$ and $c_0$, the characterization given by Corollary \ref{corban} can be simply expressed in terms of the weights.

\begin{prop}\label{prop se}
Let $X$ be the complex space $l^p$ with $1\le p<\infty$ or the complex space $c_0$, and $B_w:X\rightarrow 
X$ the weighted shift defined by $B_we_n=w_ne_{n-1}$, where $e_0=0$, $(e_n)_{n\ge 1}$ is the canonical basis and $(w_n)_{n\ge 1}$ is a sequence of non-zero scalars.
The following inequalities are equivalent:
\begin{enumerate}[\upshape (1)]
\item $\displaystyle{\sup_{n\ge 1}\sup_{E\in \text{cofin}} \inf_{x\in E\backslash \{0\}} \frac{\|B_w^nx\|}{\|x\|}\le 1}$,
\item $\displaystyle{\sup_{n\ge 1}\sup_{N\ge 1} \inf_{k\ge N} \prod_{\nu=1}^n |w_{\nu+k}|\le 1}$,
\item $\displaystyle{\sup_{n\ge 1}\sup_{N\ge 1} \inf_{k\ge N} \prod_{\nu=1}^n |w_{\nu+k}|< \infty}$.
\end{enumerate}
\end{prop}
\begin{proof}(1)$\Rightarrow$(2). We prove in fact that $\neg$(2)$\Rightarrow \neg$(1). We thus suppose that we have \[\sup_{n\ge 1}\sup_{N\ge 1} \inf_{k\ge N} \prod_{\nu=1}^n |w_{\nu+k}|> 1.\] Then there exist $C>1$, $n\ge 1$ and $N\ge 1$ such that \[\inf_{k\ge N} \prod_{\nu=1}^n |w_{\nu+k}|\ge C.\]
Letting $E_{n+N}:=\{x\in X: x_1=x_2=\cdots=x_{n+N}=0 \}$, we have for every vector $x\in E_{n+N}$,
\begin{align*}
\|B^n_wx\|&=\Big\|(0,\dots,0,\Big(\prod_{\nu=1}^n w_{\nu+N+1}\Big)x_{n+N+1},\Big(\prod_{\nu=1}^n w_{\nu+N+2}\Big)x_{n+N+2},\dots)\Big\|\\
&\ge \Big(\inf_{k\ge N} \prod_{\nu=1}^n |w_{\nu+k}|\Big)\|x\|\\
&\ge C \|x\|. 
\end{align*}
Since $E_{n+N}$ is a space of finite codimension, we deduce that we have
\[\sup_{n\ge 1}\sup_{E\in \text{cofin}} \inf_{x\in E\backslash \{0\}} \frac{\|B_w^nx\|}{\|x\|}>1. \]

(2)$\Rightarrow$(3). Evident.

(3)$\Rightarrow$(1). Let $E$ be a subspace of finite codimension. By definition, there exists  a finite-dimensional subspace $F$ such that $E\oplus F=X$. If the dimension of $F$ is $d$ then for all linearly independent vectors $x_1,\ldots, x_{d+1}$, there exist $a_{1},\dots, a_{d+1}\in \mathbb{K}$ such that
\[a_1 x_1+\cdots+a_{d+1}x_{d+1}\in E\backslash\{0\}.\]
If we suppose that we have \[\sup_{n\ge 1}\sup_{N\ge 1} \inf_{k\ge N} \prod_{\nu=1}^n |w_{\nu+k}|<K,\]
then for any $n\ge 1$, there exist different indices $i_1,\ldots,i_{d+1}$ such that for any $1\le k \le d+1$, we have \[\|B^n_w e_{n+i_{k}}\|=\prod_{\nu=1}^n |w_{\nu+i_k}|< K.\]
By the previous reasoning, there also exist $a_{1},\dots, a_{d+1}$ such that
\[a_1 e_{n+i_1}+\cdots+a_{d+1}e_{n+i_{d+1}}\in E\backslash\{0\}.\]
We thus have 
\begin{align*} \|B^n_w (a_1 e_{n+i_1}+\cdots+a_{d+1}e_{n+i_{d+1}})\| &< K \|a_1 e_{i_1}+\cdots+a_{d+1}e_{i_{d+1}}\|\\
&= K \|a_1 e_{n+i_1}+\cdots+a_{d+1}e_{n+i_{d+1}}\|. \end{align*}
We conclude that we have $\sup_{n\ge 1}\sup_{E\in \text{cofin}} \inf_{x\in E\backslash \{0\}} \frac{\|T^nx\|}{\|x\|}\le K$ and by Corollary~\ref{corban} that we have the desired inequality.
\end{proof}

In their article \cite{LeonMontes}, Le\'on and Montes have shown that a hypercyclic weighted shift $B_w$ on the complex space $l^2$ possesses a hypercyclic subspace if and only if
\[\lim_{n\rightarrow \infty}\Big(\inf_{k\ge 1} \prod_{\nu=1}^{n} |w_{\nu+k}|\Big)^{\frac{1}{n}}\le 1 \quad \text{or equivalently}\quad \sup_{n\ge 1} \inf_{k\ge 1}\prod_{\nu=1}^n |w_{\nu+k}|\le 1.\]
The supremum on $N$ in the previous result is thus surely unnecessary.

\begin{prop}\label{cond N}
Let $(w_n)_{n\ge 1}$ be a sequence of non-zero scalars.
We have
\[\sup_{n\ge 1}\sup_{N\ge 1} \inf_{k\ge N} \prod_{\nu=1}^n |w_{\nu+k}|>1 \Leftrightarrow \sup_{n\ge 1}\inf_{k\ge 1} \prod_{\nu=1}^n |w_{\nu+k}|=\infty.\]
\end{prop}
\begin{proof}
($\Leftarrow$) Obvious.

($\Rightarrow$) 
By hypothesis, there exist $C>1$, $n\ge 1$ and $N\ge 1$ such that \[\inf_{k\ge N}\prod_{\nu=1}^{n}|w_{\nu+k}|\ge C.\] We thus have for any $m\ge 1$,
\begin{equation}
\inf_{k\ge N}\prod_{\nu=1}^{mn}|w_{\nu+k}|\ge C^m. \label{eq cm}
\end{equation}
Let $K>0$. We consider an integer $m_1\ge 1$ such that $m_1n\ge N$ and $C^{m_1}>K$, and another integer $m_2\ge 1$ such that
$C^{m_2}> \frac{K}{\lambda_0}$ where $\lambda_0= \min_{1\le k\le m_1n-1}\prod_{\nu=1}^{m_1n}|w_{\nu+k}|$.
Therefore, since $m_1n\ge N$, we have, using \eqref{eq cm}, that for any $k\ge m_1n$,
\[\prod_{\nu=1}^{(m_1+m_2)n}|w_{\nu+k}|\ge C^{m_1+m_2}>K\]
and that for any $1\le k\le m_1n-1$,
\begin{align*}
\prod_{\nu=1}^{(m_1+m_2)n}|w_{\nu+k}|&= \prod_{\nu=1}^{m_1n}|w_{\nu+k}| 
\prod_{\nu=1}^{m_2n}|w_{\nu+k+m_1n}|\\
&\ge \lambda_0 C^{m_2}>K.
\end{align*}\end{proof}

Since every hypercyclic weighted shift is weakly mixing (see \cite{Salas}), it follows from the two previous results and from Corollary \ref{corban} that we have the following characterization:

\begin{theorem}\label{carac lp}
A hypercyclic weighted shift $B_w$ on the complex space $l^p$ with $1\le p<\infty$ or on the complex space $c_0$ possesses a hypercyclic subspace if and only if
\[\sup_{n\ge 1}\inf_{k\ge 1} \prod_{\nu=1}^n |w_{\nu+k}|\le1\]
or equivalently if and only if
\[\sup_{n\ge 1}\inf_{k\ge 1} \prod_{\nu=1}^n |w_{\nu+k}|<\infty.\]
\end{theorem}

By isomorphism, we can extend this result to weighted versions of complex spaces $l^p$ and $c_0$.
Let $(v_n)_{n\ge 1}$ be a strictly positive weight sequence. We define these spaces by
\begin{align*}
l^p(v)&=\Big\{(x_n)_{n\ge 1}: \sum_{n=1}^{\infty}|x_n|^pv_n<\infty\Big\} \quad \text{for } 1\le p<\infty;\\
c_0(v)&=\{(x_n)_{n\ge 1}: \lim_{n\rightarrow \infty}|x_n|v_n=0\}.\end{align*}

\begin{cor}\label{cor carac lp}
A hypercyclic weighted shift $B_w$ on the complex space $l^p(v)$ with $1\le p<\infty$ possesses a hypercyclic subspace if and only if
\[\sup_{n\ge 1}\inf_{k\ge 1} \prod_{\nu=1}^n |w_{\nu+k}|\Big(\frac{v_{k}}{v_{k+n}}\Big)^{\frac{1}{p}}\le 1\]
or equivalently if and only if
\[\sup_{n\ge 1}\inf_{k\ge 1} \prod_{\nu=1}^n |w_{\nu+k}|\Big(\frac{v_{k}}{v_{k+n}}\Big)^{\frac{1}{p}}< \infty\]
and a hypercyclic weighted shift $B_w$ on the complex space $c_0(v)$ possesses a hypercyclic subspace if and only if
\[\sup_{n\ge 1}\inf_{k\ge 1} \prod_{\nu=1}^n |w_{\nu+k}|\Big(\frac{v_{k}}{v_{k+n}}\Big)\le 1\]
or equivalently if and only if
\[\sup_{n\ge 1}\inf_{k\ge 1} \prod_{\nu=1}^n |w_{\nu+k}|\Big(\frac{v_{k}}{v_{k+n}}\Big)< \infty.\]
\end{cor}

In their article \cite{LeonMontes}, Le\'on and Montes also show that every hypercyclic bilateral weighted shift on the complex space $l^2(\mathbb{Z})$ possesses a hypercyclic subspace. We can extend this result to complex spaces $l^p(\mathbb{Z})$ with $1\le p<\infty$, to the complex space $c_0(\mathbb{Z})$ and afterwards, to complex weighted spaces $l^p(v,\mathbb{Z})$ and $c_0(v,\mathbb{Z})$.

\begin{theorem}
Every hypercyclic bilateral weighted shift on the complex space $l^p(\mathbb{Z})$ with $1\le p<\infty$ or on the complex space $c_0(\mathbb{Z})$ possesses a hypercyclic subspace.
\end{theorem}
\begin{proof}
Let $B_w$ be a hypercyclic bilateral weighted shift defined by $B_w e_n=w_ne_{n-1}$ for any $n\in \mathbb{Z}$. By the characterization of hypercyclic bilateral weighted shifts (see \cite{Karl2}, \cite{Salas}), we know that there exists an increasing sequence $(n_k)_{k\ge 1}$ of positive integers such that for any $j\in \mathbb{Z}$,
\begin{equation}
\lim_{k\rightarrow \infty} \prod_{\nu=0}^{n_k-1}|w_{j-\nu}|=0.
\label{max 0}
\end{equation}
Suppose that $B_w$ does not possess any hypercyclic subspace. Then, by Corollary \ref{corban}, we have
\[\sup_{n\ge 1}\sup_{E\in\text{cofin}}\inf_{x\in E\backslash \{0\}} \frac{\|B_w^nx\|}{\|x\|}> 1.\]
It is not difficult to see that, as in Proposition \ref{prop se}, the above condition is equivalent to having
\[\sup_{n\ge 1}\sup_{N\ge 1} \inf_{|k|\ge N} \prod_{\nu=1}^n |w_{\nu+k}|> 1.\]
There thus exist $C>1$, $n\ge 1$ and $N\ge 1$ such that for any $j\ge N$, we have
\begin{equation}
\prod_{\nu=0}^{n-1} |w_{-j-\nu}|>C.
\label{min}
\end{equation}
Since $B_w$ is continuous, there also exists $K>0$ such that $\sup_{l\in \mathbb{Z}}|w_l|<K$ and if $n_k=m_kn+p$ with $0\le p\le n-1$ then we have by using \eqref{min}
\[\prod_{\nu=0}^{n_k-1}|w_{-N-\nu}|=\frac{\prod_{l=0}^{m_k}\prod_{\nu=0}^{n-1}|w_{-N-ln-\nu}|}{\prod_{\nu=p}^{n-1}|w_{-N-m_kn-\nu}|}>\frac{C^{m_k+1}}{K^{n-p}}.\]
This is a contradiction with the equation \eqref{max 0}. The operator $B_w$ possesses thus a hypercyclic subspace.
\end{proof}

\begin{cor}
Every hypercyclic bilateral weighted shift on the complex space $l^p(v,\mathbb{Z})$ with $1\le p<\infty$ or on the complex space $c_0(v,\mathbb{Z})$ possesses a hypercyclic subspace.
\end{cor}

\section{Weighted shifts on certain Köthe sequence spaces}
Let $A=(a_{j,k})_{j,k\ge 1}$ be a matrix such that for any $j,k\ge 1$, we have $a_{j,k}>0$ and $a_{j,k}\le a_{j+1,k}$.
We define the (real or complex) Köthe sequence spaces $\lambda^p(A)$ with $1\le p<\infty$ and $c_0(A)$ by
\begin{align*}
\lambda^p(A)&:=\Big\{(x_k)_k\in \mathbb{K}^{\mathbb{N}} : p_j((x_k)_k)=\Big(\sum_{k=1}^{\infty}|x_ka_{j,k}|^p\Big)^{\frac 1 p}<\infty, \ j\ge 1\Big\},\\
c_0(A)&:=\{(x_k)_k\in \mathbb{K}^{\mathbb{N}}: \lim_{k\rightarrow \infty}|x_k|a_{j,k}=0, \ j\ge 1\} \text{ with } p_j((x_k)_k)= \max_k|x_k|a_{j,k}.
\end{align*}

These spaces are Fréchet spaces with a continuous norm and the sequences of norms $(p_j)$ are increasing (see \cite{Meise} for more details about Köthe sequence spaces). Therefore we deduce from Corollary \ref{corfre} the following result:

\begin{theorem}\label{Kothe 1}
Let $B_w:X\rightarrow X$ be a weighted shift, where $X=\lambda^p(A)$ with $1\le p<\infty$ or $X=c_0(A)$.
If there exists an integer $J\ge 1$ such that for any $j\ge 1$, we have
\[\sup_{n\ge1}\sup_{N\ge 1}\inf_{k\ge N}\frac{p_{J}(B_w^ne_k)}{p_{j}(e_{k})}> 1\]
then $B_w$ does not possess any hypercyclic subspace.
\end{theorem}
\begin{proof}
By hypothesis, there exists $J\ge 1$ such that for any $j\ge 1$, there exist $C_j>1$, $n_j\ge 1$, $N_j\ge 1$ such that we have for any $k\ge N_j$,
\[\frac{p_{J}(B_w^{n_j}e_k)}{p_{j}(e_{k})}> C_j.\]
The condition of Corollary \ref{corfre} is then satisfied because for any $j\ge 1$, for any $x \in E_{N_j}:=\{x\in X:x=\sum_{k=N_j}^{\infty}x_ke_k\}$, $x\ne 0$, we have if $X=\lambda^p(A)$,
\begin{align*}
p_{J}(B^{n_j}_{w} x)^p&= p_{J}\Big(B^{n_j}_w\Big(\sum_{k=N_j}^{\infty}x_ke_k\Big)\Big)^p
= \sum_{k=N_j}^{\infty}|x_{k}|^p p_{J}(B_w^{n_j}e_k)^p\\
&> \sum_{k=N_j}^{\infty}|x_{k}|^p C^p_j p_{j}(e_{k})^p
= C^p_jp_j(x)^p,
\end{align*}
and we have if $X=c_0(A)$,
\begin{align*}
p_{J}(B^{n_j}_{w} x)&= p_{J}\Big(B^{n_j}_w\Big(\sum_{k=N_j}^{\infty}x_ke_k\Big)\Big)
= \max_{k\ge N_j}|x_{k}| p_{J}(B_w^{n_j}e_k)\\
&> \max_{k\ge N_j}|x_{k}| C_j p_{j}(e_{k})
= C_jp_j(x).
\end{align*}
\end{proof}

For the other implication, we cannot proceed as in the proof of Proposition \ref{prop se} because unlike Corollary \ref{corban}, Corollary \ref{corfre} does not give us an equivalence. We will thus prove that if for any $j\ge 1$, there exists $m_j\ge 1$ such that 
\[\sup_{n\ge1}\sup_{N\ge 1}\inf_{k\ge N}\frac{p_{j}(B_w^ne_k)}{p_{m_j}(e_{k})}<\infty\]
then $B_w$ satisfies the condition of Theorem \ref{thm M}. For this purpose, we need the following result whose idea of proof is the same as for Theorem \ref{thmfrechet}.
This result is stated for Fréchet sequence spaces which are Fréchet spaces of sequences that are continuously embedded in $\omega$. Obviously Köthe sequence spaces are Fréchet sequence spaces.
\begin{lemma}\label{lemma 4}
Let $X$ be a Fréchet sequence space with a continuous norm, $(p_j)$ an increasing sequence of norms defining the topology of $X$ and $B_w:X\rightarrow X$ a weighted shift.
The following assertions are equivalent:
\begin{enumerate}[\upshape (i)]
\item there exists $J\ge 1$ such that for any $j\ge 1$, we have
\[\sup_{n\ge1}\sup_{N\ge 1}\inf_{k\ge N}\frac{p_{J}(B_w^ne_{k})}{p_{j}(e_{k})}=\infty;\]
\item there exists $J\ge 1$ such that for any $j\ge 1$, we have
\[\sup_{n\ge1}\sup_{N\ge 1}\inf_{k\ge N}\max_{m\le n}\frac{p_{J}(B_w^me_{k})}{p_{j}(e_{k})}> 1.\]
\end{enumerate}
\end{lemma}
\begin{proof}
The implication (i)$\Rightarrow$(ii) is evident. We show the implication (ii)$\Rightarrow$(i). We know by hypothesis that there exist $j_0\ge1$ such that for any $j\ge 1$, we have
\[\sup_{n\ge1}\sup_{N\ge 1}\inf_{k\ge N}\max_{m\le n}\frac{p_{j_0}(B_w^me_{k})}{p_{j}(e_{k})}> 1.\]
There thus exist $C_{j_0}>1$, $n_{j_0}\ge 1$, $N_{j_0}\ge 1$ such that for any $k\ge N_{j_0}$, there exists $1\le m\le n_{j_0}$ such that we have \begin{equation}\frac{p_{j_0}(B_w^{m}e_{k})}{p_{j_0}(e_{k})}> C_{j_0}.\label{pj0}\end{equation}
Let $j\ge 1$. There are also $C_{j}>1$, $n_{j}\ge1$, $N_{j}\ge 1$ such that for any $k\ge N_{j}$, there exists $m\le n_{j}$ such that we have
\begin{equation}
\frac{p_{j_0}(B_w^me_{k})}{p_{j}(e_{k})}> C_{j}.
\label{pj}
\end{equation}
Let $k\ge N_{j_0}+N_j+ln_{j_0}+n_j$ and $n=ln_{j_0}+n_j$ with $l\ge1$.
We deduce from \eqref{pj} that there exists $m_1\le n_j$ such that
\[\frac{p_{j_0}(B_w^{m_1}e_{k})}{p_{j}(e_{k})}> C_{j}\]
and we deduce by repeatedly applying \eqref{pj0} to $B_w^{m_1}e_k$ that there exists $n\le m_2< n+n_{j_0}$ such that
\[\frac{p_{j_0}(B_w^{m_2}e_{k})}{p_{j_0}(B_w^{m_1}e_{k})}> C^l_{j_0}.\]
Since $X$ is a Fréchet sequence space and $B_w$ maps $X$ into itself, the weighted shift $B_w$ is continuous. There thus exist $K>0$ and $J\ge 1$ such that for any $0\le p< n_{j_0}$, we have
\[\text{\ } p_{j_0}(B_w^px)\le K p_{J}(x)\quad\text{for any $x\in X$}\]
and therefore for any $0\le p< n_{j_0}$, for any $m\ge p$, we have
\begin{equation*}
p_J(B_w^{m-p} x)\ge \frac{1}{K}p_{j_0}(B_w^{m}x) \quad\text{for any $x\in X$}.
\end{equation*}
It follows that we have
\begin{align*}
p_{J}(B_w^ne_k)&= 
p_{J}(B_w^{m_2-(m_2-n)}e_k)
\ge \frac{1}{K}p_{j_0}(B_w^{m_2}e_k)\\
&> \frac{C_{j_0}^l}{K}p_{j_0}(B_w^{m_1}e_k)
> \frac{C_{j_0}^lC_j}{K}p_j(e_k).
\end{align*}
Since the choice of $J$ does not depend of $j$ and $C_{j_0}^l\rightarrow \infty$ when $l\rightarrow \infty$, we have the desired result.
\end{proof}

To prove that a hypercyclic weighted shift $B_w$ possesses a hypercyclic subspace if for any $j\ge 1$, there exists $m_j\ge 1$ such that 
\[\sup_{n\ge1}\sup_{N\ge 1}\inf_{k\ge N}\frac{p_{j}(B_w^ne_k)}{p_{m_j}(e_{k})}<\infty,\]
we need to suppose an additional condition on the norms $(p_j)_{j\ge 1}$ and thus on the matrix $A$. This condition will permit us to transpose the inequalities for a couple of seminorms $(p_J, p_{m})$ to others couples $(p_j,p_{m_j})$.

\begin{theorem}\label{Kothe 2}
Let $B_w:X\rightarrow X$ be a hypercyclic weighted shift, where $X=\lambda^p(A)$ with $1\le p<\infty$ or $X=c_0(A)$.
Suppose that there exists $J\ge 1$ such that for any $m\ge 1$, for any $j\ge 1$, there exists $m_j\ge 1$ such that \begin{equation*}
\sup_{n\ge 0}\limsup_{k>n} \frac{p_{j}(e_{k-n})p_{m}(e_{k})}{p_{J}(e_{k-n})p_{m_{j}}(e_{k})}<\infty.
\end{equation*}
If there exists $m\ge 1$ such that we have
\[\sup_{n\ge1}\sup_{N\ge 1}\inf_{k\ge N}\frac{p_{J}(B_w^ne_k)}{p_m(e_{k})}<\infty\]
then $B_w$ possesses a hypercyclic subspace.
\end{theorem}
\begin{proof}
By our assumptions, we know that there exists $m\ge 1$ such that for any $j\ge 1$, there exists $m_j\ge 1$ such that we have
\[\sup_{n\ge1}\sup_{N\ge 1}\inf_{k\ge N}\frac{p_{j}(B_w^ne_k)}{p_{m_{j}}(e_{k})}=
\sup_{n\ge1}\sup_{N\ge n+1}\inf_{k\ge N}\frac{p_{J}(B_w^ne_k)}{p_{m}(e_{k})}
\frac{p_{j}(e_{k-n})p_{m}(e_{k})}{p_{J}(e_{k-n})p_{m_{j}}(e_{k})}
 <\infty.\]
By the previous lemma, the above condition is equivalent to having, for any $j\ge 1$, the existence of an integer $j'\ge 1$ such that
\[\sup_{n\ge1}\sup_{N\ge 1}\inf_{k\ge N}\max_{i\le n}\frac{p_{j}(B_w^ie_k)}{p_{j'}(e_{k})}\le 1.\]
In particular, we can choose $m_J\ge 1$ such that 
\begin{equation}
\sup_{n\ge1}\sup_{N\ge 1}\inf_{k\ge N}\max_{i\le n}\frac{p_{J}(B_w^ie_k)}{p_{m_{J}}(e_{k})}\le 1.
\label{maj pj0}
\end{equation}
Let $(n_l)$ be an increasing sequence of integers such that $B_w$ satisfies the Hypercyclicity Criterion for $(n_l)$. Such a sequence exists because every hypercyclic weighted shift is weakly mixing (see \cite{Karl2}). Then, using the relation \eqref{maj pj0}, we have that for any $l\ge 1$, for any $N\ge 1$, there exists $k\ge N$ such that for any $l'\le l$, we have \[\frac{p_{J}(B_w^{n_{l'}}e_k)}{p_{m_{J}}(e_{k})}\le 2.\]
If we let $C_J=2$ and for any $j\ne J$, we let $C_j$ be a constant and $m_j$ an integer satisfying
\begin{equation*}
\sup_n\limsup_{k>n} \frac{p_{j}(e_{k-n})p_{m_{J}}(e_{k})}{p_{J}(e_{k-n})p_{m_{j}}(e_{k})}<C_j,
\end{equation*}
then it is not difficult to construct an increasing sequence of integers $(k_l)_{l\ge 1}$ such that for any $l\ge 1$, we have $n_{l}<k_l$ and for any $l\ge 1$, we have
\begin{align}
&\text{for any\ } l'\le l, \quad\frac{p_{J}(B_w^{n_{l'}}e_{k_l})}{p_{m_{J}}(e_{k_l})}\le 2;\label{eq esp 1}\\
&\text{for any\ } j\le l'\le l,\quad \frac{p_{j}(e_{k_l-n_{l'}})p_{m_{J}}(e_{k_l})}{p_{J}(e_{k_l-n_{l'}})p_{m_{j}}(e_{k_l})}<C_{j}\label{eq esp 2}.
\end{align}
Therefore we have by multiplying \eqref{eq esp 1} and \eqref{eq esp 2} that for any $l\ge 1$, for any $j\le l'\le l$,
\[\frac{p_{j}(B^{n_{l'}}_w e_{k_l})}{p_{m_{j}}(e_{k_l})}\le 2C_{j}.\]
If for any $l'\ge 1$, we let $M_{l'}$ be the closed linear span of $(e_{k_{l}})_{l\ge l'}$ then for any $j\le l'$, for any $x\in M_{l'}$, we have if $X=\lambda^p(A)$,
\begin{align*}
p_{j}(B^{n_{l'}}_wx)^p &= p_{j}\Big(B^{n_{l'}}_w\Big(\sum_{l=l'}^{\infty}x_{k_l} e_{k_l}\Big)\Big)^p
= \sum_{l=l'}^{\infty}|x_{k_l}|^p p_{j}(B^{n_{l'}}_w e_{k_l})^p\\
&\le \sum_{l=l'}^{\infty}|x_{k_l}|^p (2C_{j})^p p_{m_{j}}(e_{k_l})^p= (2C_{j})^p p_{m_{j}}(x)^p
\end{align*}
and we have if $X=c_0(A)$,
\begin{align*}
p_{j}(B^{n_{l'}}_wx) &= p_{j}\Big(B^{n_{l'}}_w\Big(\sum_{l=l'}^{\infty}x_{k_l} e_{k_l}\Big)\Big)
= \max_{l\ge l'}|x_{k_l}| p_{j}(B^{n_{l'}}_w e_{k_l})\\
&\le \max_{l\ge l'}|x_{k_l}| (2C_{j}) p_{m_{j}}(e_{k_l})= 2C_{j} p_{m_{j}}(x).
\end{align*}
Theorem \ref{thm M} implies the desired result.
\end{proof}

Theorems \ref{Kothe 1} and \ref{Kothe 2} allow us to characterize for the first time which weighted shifts on certain Fréchet spaces possess a hypercyclic subspace.
Before stating this characterization, we remark that we can simplify this one as in Proposition \ref{cond N}.
\begin{prop}\label{4.4}
Let $(w_n)_{n\ge 1}$ be a sequence of non-zero scalars and $(a_{j,k})_{j,k\ge 1}$ a family of positive numbers. For any $j\ge 1$, we have
\begin{align*} &\forall m\ge 1,\quad \sup_{n\ge1}\sup_{N\ge 1}\inf_{k\ge N}\frac{\prod_{\nu=1}^{n}|w_{\nu + k}|a_{j,k}}{a_{m,n+k}}> 1\\
 &\Leftrightarrow \forall m\ge 1,\quad \sup_{n\ge1}\inf_{k\ge 1}\frac{\prod_{\nu=1}^{n}|w_{\nu + k}|a_{j,k}}{a_{m,n+k}}=\infty.\end{align*}
\end{prop} 
\begin{proof}
($\Leftarrow$) Evident.

($\Rightarrow$) Let $j\ge 1$.
By hypothesis, for any $m\ge 1$, there exists $C_m>1$, $n_m\ge 1$ and $N_m\ge 1$ such that
\[\inf_{k\ge N_m}\frac{\prod_{\nu=1}^{n_m}|w_{\nu + k}|a_{j,k}}{a_{m,n_m+k}}>C_m.\]
In particular, we have for any $l\ge 1$,
\begin{equation*}
\inf_{k\ge N_{j}}\frac{\prod_{\nu=1}^{ln_{j}}|w_{\nu + k}|a_{j,k}}{a_{j,ln_{j}+k}}> C_j^l.
\end{equation*}
Let $m\ge 1$ and $K>0$. We consider an integer $l_1\ge 1$ such that $l_1n_{j}\ge\max(N_{j},N_m)$ and $C_j^{l_1}>K$, and another integer $l_2\ge 1$ such that
$\lambda_0 C_j^{l_2}> K$ where \[\lambda_0= \min_{0\le k\le l_1n_{j}-1}\frac{\prod_{\nu=1}^{l_1n_{j}}|w_{\nu + k}|a_{j,k}}{a_{j,l_1n_{j}+k}}.\]
Therefore, since $l_1n_{j}\ge\max(N_{j},N_m)$, we have for any $k\ge l_1n_{j}$,
\begin{align*}
\frac{\prod_{\nu=1}^{(l_1+l_2)n_{j}+n_m}|w_{\nu + k}|a_{j,k}}{a_{m,(l_1+l_2)n_{j}+n_m+k}}&=
\frac{\prod_{\nu=1}^{(l_1+l_2)n_{j}}|w_{\nu + k}|a_{j,k} \prod_{\nu=1}^{n_m}|w_{\nu+(l_1+l_2)n_{j}+k}|} {a_{m,(l_1+l_2)n_{j}+n_m+k}}\\
&>C_j^{l_1+l_2}\frac{\prod_{\nu=1}^{n_m}|w_{\nu+(l_1+l_2)n_{j}+k}|a_{j,(l_1+l_2)n_{j}+k}}{a_{m,(l_1+l_2)n_{j}+n_m+k}}\\
&>C_j^{l_1+l_2}C_m>K
\end{align*}
and for any $1\le k\le l_1n_{j}-1$,
\begin{align*}
\frac{\prod_{\nu=1}^{(l_1+l_2)n_{j}+n_m}|w_{\nu + k}|a_{j,k}}{a_{m,(l_1+l_2)n_{j}+n_m+k}}
&= \frac{\prod_{\nu=1}^{l_1n_{j}}|w_{\nu + k}|a_{j,k}}{a_{j,l_1n_{j}+k}}\cdot \frac{\prod_{\nu=1}^{l_2n_{j}}|w_{\nu+l_1n_{j}+k}|a_{j,l_1n_{j}+k}}{a_{j,(l_1+l_2)n_{j}+k}}\\
&\quad\quad\cdot\frac{\prod_{\nu=1}^{n_m}|w_{\nu+(l_1+l_2)n_{j}+k}|a_{j,(l_1+l_2)n_{j}+k}}{a_{m,(l_1+l_2)n_{j}+n_m+k}}\\
&>\lambda_0 C_j^{l_2}C_m>K.
\end{align*}
\end{proof}
We obtain the following characterization:
\begin{theorem}
Let $B_w:X\rightarrow X$ be a hypercyclic weighted shift, where $X=\lambda^p(A)$ with $1\le p<\infty$ or $X=c_0(A)$.
Suppose that there exists $J\ge 1$ such that for any $m\ge 1$, for any $j\ge 1$, there exists $m_j\ge 1$ such that
\begin{equation}
\sup_{n\ge 0}\limsup_k \frac{a_{j,k}a_{m,n+k}}{a_{J,k}a_{m_{j},n+k}}<\infty.
\label{cond B}
\end{equation}
Then $B_w$ possesses a hypercyclic subspace if and only if there exists $m\ge 1$ such that
\begin{equation*}
\sup_{n\ge1}\inf_{k\ge 1}\frac{\prod_{\nu=1}^{n}|w_{\nu + k}|a_{J,k}}{a_{m,n+k}}\le 1
\end{equation*}
or equivalently if and only if
\begin{equation*}
\sup_{n\ge1}\inf_{k\ge 1}\frac{\prod_{\nu=1}^{n}|w_{\nu + k}|a_{J,k}}{a_{m,n+k}}<\infty.
\end{equation*}
\end{theorem}

In particular, this theorem can be applied in the case of spaces $l^p(v)$ and $c_0(v)$ if we consider $a_{j,k}=(v_k)^{\frac{1}{p}}$ and it extends thereby Corollary \ref{cor carac lp} to the case of real spaces.

\begin{cor}\label{cor carac lp 2}
A hypercyclic weighted shift on the real or complex space $l^p(v)$ with $1\le p<\infty$ possesses a hypercyclic subspace if and only if
\[\sup_{n\ge 1}\inf_{k\ge 1} \prod_{\nu=1}^n |w_{\nu+k}|\Big(\frac{v_{k}}{v_{k+n}}\Big)^{\frac{1}{p}}\le 1\]
or equivalently if and only if
\[\sup_{n\ge 1}\inf_{k\ge 1} \prod_{\nu=1}^n |w_{\nu+k}|\Big(\frac{v_{k}}{v_{k+n}}\Big)^{\frac{1}{p}}< \infty,\]
and a hypercyclic weighted shift on the real or complex space $c_0(v)$ possesses a hypercyclic subspace if and only if
\[\sup_{n\ge 1}\inf_{k\ge 1} \prod_{\nu=1}^n |w_{\nu+k}|\Big(\frac{v_{k}}{v_{k+n}}\Big)\le 1\]
or equivalently if and only if
\[\sup_{n\ge 1}\inf_{k\ge 1} \prod_{\nu=1}^n |w_{\nu+k}|\Big(\frac{v_{k}}{v_{k+n}}\Big)< \infty.\]
\end{cor}

An example of Fréchet space satisfying condition \eqref{cond B}
is the space of entire functions. This space, denoted by $H(\mathbb{C})$, can be seen as the Köthe sequence space $\lambda^1(A)$ with $a_{j,n}=j^n$, and condition \eqref{cond B} is then satisfied for $J=1$ and $m_j=2jm$ because for any $m, j\ge 1$, we have
\[\limsup_k \frac{a_{j,k}a_{m,n+k}}{a_{J,k}a_{m_{j},n+k}}=\limsup_k \frac{j^km^{n+k}}{J^km_{j}^{n+k}}=\limsup_k \frac{j^km^{n+k}}{(2j)^{n+k}m^{n+k}}=0.\]

\begin{cor}
Let $B_w:H(\mathbb{C})\rightarrow H(\mathbb{C})$ be a hypercyclic weighted shift defined by $B_wz^n=w_nz^{n-1}$. The operator $B_w$ possesses a hypercyclic subspace if and only if there exists $m\ge 1$ such that
\begin{equation}
\sup_{n\ge1}\inf_{k\ge 0}\frac{\prod_{\nu=1}^{n}|w_{\nu + k}|}{m^{n+k}}\le 1.
\label{cond hc}
\end{equation}
\end{cor}
\begin{remark}In 2010, Shkarin \cite{Shkarin} has shown that the differentiation operator on the space of entire functions possesses a hypercyclic subspace.
In fact, the differentiation operator can be seen as the weighted shift $B_w$ with $w_n=n$ and this operator satisfies condition \eqref{cond hc}:
\begin{equation}
\sup_{n\ge1}\inf_{k\ge 0}\frac{\prod_{\nu=1}^{n}|w_{\nu + k}|}{2^{n+k}}\le \sup_{n\ge1}\inf_{k\ge 0}\frac{(k+n)^n}{2^{n+k}}=0.\label{diff}
\end{equation}
We have thus improved the result of Shkarin by giving a characterization of weighted shifts on $H(\mathbb{C})$ with hypercyclic subspaces.
\end{remark}

Another important example of a Fréchet space satisfying condition \eqref{cond B}
is the space $s$ of rapidly decreasing sequences, which is the space $\lambda^1(A)$ with $a_{j,k}=k^j$. This space is isomorphic to $ C^{\infty}([0,1])$ (see \cite[Example 29.4]{Meise}).

\begin{cor}
Let $B_w:s\rightarrow s$ be a hypercyclic weighted shift. The operator $B_w$ possesses a hypercyclic subspace if and only if there exists $m\ge 1$ such that
\begin{equation*}
\sup_{n\ge1}\inf_{k\ge 0}\frac{\prod_{\nu=1}^{n}|w_{\nu + k}|}{(n+k)^m}\le 1.
\end{equation*}
\end{cor}

\begin{remark}
A sufficient condition on $A$ to satisfy \eqref{cond B} is that for any $j,k\ge1$, we have $a_{j,k}\le a_{j,k+1}$ and for any $j\ge 1$, there exists $m_j\ge 1$ such that \[\sup_k \frac{a_{j,k}^2}{a_{m_j,k}}<\infty.\]
On the other hand, some matrices $A$ do not satisfy \eqref{cond B}. For example, the matrix $A$ with $a_{j,k}=k^{1-\frac{1}{j}}$ does not satisfy \eqref{cond B}.
\end{remark}

As in the case of the complex spaces $l^p(\mathbb{Z})$ and $c_0(\mathbb{Z})$, we can prove the existence of a hypercyclic subspace for every hypercyclic bilateral weighted shifts on certain real or complex Köthe sequence spaces. Let $A=(a_{j,k})_{j\ge 1,k\in \mathbb{Z}}$ be a matrix such that for any $j\ge 1, k\in \mathbb{Z}$, we have $a_{j,k}>0$ and $a_{j,k}\le a_{j+1,k}$. We define the spaces $\lambda^p(A,\mathbb{Z})$ with $1\le p<\infty$ and $c_0(A,\mathbb{Z})$ by
\begin{align*}
\lambda^p(A,\mathbb{Z})&:=\Big\{(x_k)_k\in \mathbb{K}^{\mathbb{Z}} : p_j((x_k)_k)=\Big(\sum_{k\in \mathbb{Z}}|x_ka_{j,k}|^p\Big)^{\frac 1 p}<\infty, \ j\ge 1\Big\},\\
c_0(A,\mathbb{Z})&:=\{(x_k)_k\in \mathbb{K}^{\mathbb{Z}} : \lim_{k\rightarrow \pm \infty}|x_k|a_{j,k}=0, \ j\ge 1\} \text{ with } p_j((x_k)_k)= \max_{k\in \mathbb{Z}}|x_k|a_{j,k}.
\end{align*}
Suppose that $X=\lambda^p(A,\mathbb{Z})$ or $c_0(A,\mathbb{Z})$ and that there exists $J\ge 1$ such that for any $m\ge 1$, for any $j\ge 1$, there exists $m_j\ge 1$ such that
\begin{equation*}
\sup_{n\ge 0}\limsup_{k\rightarrow -\infty} \frac{a_{j,k}a_{m,n+k}}{a_{J,k},a_{m_{j},n+k}}<\infty.
\end{equation*} 
Proceeding in the same way as for unilateral shifts, one can show that for any hypercyclic bilateral weighted shift $B_w:X\rightarrow X$, if there exists $m\ge 1$ such that we have
\[\sup_{n\ge1}\sup_{N\ge 0}\inf_{k\le -N}\frac{\prod_{\nu=1}^{n}|w_{\nu + k}|a_{J,k}}{a_{m,n+k}}\le 1\] then $B_w$ possesses a hypercyclic subspace. We show that if $B_w$ is hypercyclic then $B_w$ satisfies this condition and thus that every hypercyclic bilateral weighted shift possesses a hypercyclic subspace.

\begin{theorem}
Let $X=\lambda^p(A,\mathbb{Z})$ with $1\le p<\infty$ or $X=c_0(A,\mathbb{Z})$.
Suppose that there exists $J\ge 1$ such that for any $m\ge 1$, for any $j\ge 1$, there exists $m_j\ge 1$ such that
\begin{equation*}
\sup_{n\ge 0}\limsup_{k\rightarrow -\infty} \frac{a_{j,k}a_{m,n+k}}{a_{J,k}a_{m_{j},n+k}}<\infty.
\end{equation*}
Then every hypercyclic bilateral weighted shift possesses a hypercyclic subspace.
\end{theorem}
\begin{proof}
By the characterization of hypercyclic bilateral weighted shifts (see \cite{Karl2}), we know that there exists  an increasing sequence of integers $(n_k)_{k\ge 1}$ such that for any $j\in \mathbb{Z}$,
\begin{equation}
\lim_{k\rightarrow \infty} \prod_{\nu=0}^{n_k-1}|w_{j-\nu}|a_{J,j-n_k}=0.
\label{max 02}
\end{equation}
Suppose that $B_w$ does not possess any hypercyclic subspace. We then have for any $m\ge 1$, 
\[\sup_{n\ge1}\sup_{N\ge 0}\inf_{k\le -N}\frac{\prod_{\nu=1}^{n}|w_{\nu + k}|a_{J,k}}{a_{m,n+k}}>1.\]
In particular, there thus exist $C>1$, $n\ge 1$ and $N\ge 0$ such that for any $k\le -N$, we have
\begin{equation}
\frac{\prod_{\nu=1}^{n}|w_{k+\nu}|a_{J,k}}{a_{J,k+n}}>C.
\label{min2}
\end{equation}
We fix $\displaystyle\lambda_0=\max_{0\le p \le n-1} \frac{\prod_{\nu=1}^{n-p}|w_{-N+\nu}|}{a_{J,-N+n-p}}$. If $n_k=m_kn+p$ with $0\le p\le n-1$ and $m_k\ge 1$ then we have by using \eqref{min2}
\begin{align*}
\prod_{\nu=0}^{n_k-1}|w_{-N-\nu}|a_{J,-N-n_k}&=\frac{\prod_{l=0}^{m_k}\prod_{\nu=1}^{n}|w_{-N-n_k+ln+\nu}|a_{J,-N-n_k}}{\prod_{\nu=1}^{n-p}|w_{-N+\nu}|}\\
&>\frac{C^{m_k+1}a_{J,-N+n-p}}{\prod_{\nu=1}^{n-p}|w_{-N+\nu}|}\\
&>\frac{C^{m_k+1}}{\lambda_0}.\end{align*}
This is a contradiction with the equation \eqref{max 02}. We conclude that the operator $B_w$ possesses a hypercyclic subspace.
\end{proof}

\begin{cor}
Every hypercyclic bilateral weighted shift on the real or complex space $l^p(v,\mathbb{Z})$ or on the real or complex space $c_0(v,\mathbb{Z})$ possesses a hypercyclic subspace.
\end{cor}

\section{Functions of weighted shifts}

We finish by looking at the operators of the form $P(B_w)$, where $P
$ is a non-constant polynomial and $B_w$ is a weighted shift. In particular, we prove that for every non-constant polynomial $P$, the operator $P(D)$ possesses a hypercyclic subspace on $H(\mathbb{C})$ where $D$ is the differentiation operator. It was already known that $D$ possesses a hypercyclic subspace (see \cite{Shkarin}) and that if $\phi$ is an entire function of exponential type which is not a polynomial then $\phi(D)$ possesses a hypercyclic subspace (see \cite{Petersson}). However the question was open in the case where $\phi$ is an arbitrary non-constant polynomial. 

\begin{lemma}\label{lem 0}
Let $X$ be a Fréchet sequence space with a continuous norm, $(p_j)$ an increasing sequence of norms defining the topology of $X$ and $B_w:X\rightarrow X$ a weighted shift. The following assertions are equivalent:
\begin{enumerate}[\upshape (i)]
\item there exists $J\ge 1$ such that for any $j\ge 1$, we have
\[\sup_{n\ge1}\sup_{N\ge 1}\inf_{k\ge N}\frac{p_{J}(B_w^ne_k)}{p_{j}(e_{k})}>0;\]
\item there exists $J\ge 1$ such that for any $j\ge 1$, we have
  \[\sup_{n\ge1}\sup_{N\ge 1}\inf_{k\ge N}\max_{1\le m\le n}\frac{p_{J}(B_w^me_k)}{p_{j}(e_{k})}>0.\]
\end{enumerate}
\end{lemma}
\begin{proof}
The implication (i)$\Rightarrow$(ii) is obvious.
On the other hand, if we suppose that (ii) is satisfied then
there exists $j_0$ such that for any $j\ge 1$, there exist $0< \varepsilon_j\le 1$ and $n_j,N_j\ge 1$ such that for any $k\ge N_j$, we have for some $1\le m\le n_j$,
\[p_{j_0}(B_w^me_k)\ge \varepsilon_j p_j(e_k).\]
Using the previous property repeatedly for $j=j_0$, we deduce that for any $n\ge 1$, any $1\le m\le n$ and for any $k\ge N_{j_0}+n$, there exists $n\le m'< n+n_{j_0}$ such that
\[p_{j_0}(B_w^{m'}e_k)\ge \varepsilon_{j_0}^{n} p_{j_0}(B_w^me_k).\]
By continuity, there also exist $J\ge 1$ and $K>0$ such that for any $0\le m< n_{j_0}$, for any $x\in X$, we have \[p_{j_0}(B_w^mx)\le Kp_J(x).\]
Let $j\ge 1$. Let $N=N_j+N_{j_0}+n_j$.
We deduce that, for any $k\ge N$, there exist $1\le m_1\le n_j$ and $n_j\le m_2<n_j+n_{j_0}$ such that
\begin{align*}
p_J(B_w^{n_j}e_k)&=p_J(B_w^{m_2-(m_2-n_j)}e_k)\ge \frac{1}{K}p_{j_0}(B_w^{m_2}e_k)\\
&\ge \frac{\varepsilon_{j_0}^{n_j}}{K}p_{j_0}(B_w^{m_1}e_k)\ge \frac{\varepsilon_{j_0}^{n_j}\varepsilon_j}{K}p_j(e_k).
\end{align*}
\end{proof}

\begin{theorem}\label{thm poly}
Let $B_w:X\rightarrow X$ be a weighted shift, where $X=\lambda^p(A)$ with $1\le p<\infty$ or $X=c_0(A)$, and $P$ a non-constant polynomial such that $P(B_w)$ satisfies the Hypercyclicity Criterion.
Suppose that there exists $J\ge 1$ such that for any $m\ge 1$, for any $j\ge 1$, there exists $m_j\ge 1$ such that \begin{equation*}
\sup_{n\ge 0}\limsup_{k>n} \frac{p_{j}(e_{k-n})p_{m}(e_{k})}{p_{J}(e_{k-n})p_{m_{j}}(e_{k})}<\infty.
\end{equation*}
If there exists $m\ge 1$ such that for any $n\ge 1$,
\[\inf_{k\ge n+1}\frac{p_{J}(B_w^ne_k)}{p_m(e_{k})}=0\]
and if one of the following two conditions is satisfied:
\begin{enumerate}[\upshape (1)]
 \item $P$ has a constant term $|c_0|\le 1$;
 \item there exists $m\ge 1$ such that $\lim_k \frac{p_J(e_k)}{p_m(e_k)}=0$,
\end{enumerate}
then $P(B_w)$ possesses a hypercyclic subspace.
\end{theorem}
\begin{proof}
Suppose that $P$ is a polynomial of degree $d$ with constant term $c_0$ such that $P(B_w)$ satisfies the Hypercyclicity Criterion for the sequence $(n_k)$. For any $n\ge 1$, we thus have $P^n(B_w)=\sum_{i=0}^{nd}c_i^{(n)}B_w^i$ for some constants $c_i^{(n)}$ and we fix 
\[K_n=\max\{|c_i^{(k)}|:i\le kd, k\le n\}.\]
By assumption, we know that there exists $m\ge 1$ such that for any $j\ge 1$, there exists $m_j\ge 1$ such that
\[\sup_{n\ge1}\sup_{N\ge 1}\inf_{k\ge N}\frac{p_{j}(B_w^ne_k)}{p_{m_{j}}(e_{k})}=
\sup_{n\ge1}\sup_{N>n}\inf_{k\ge N}\frac{p_{J}(B_w^ne_k)}{p_{m}(e_{k})}
\frac{p_{j}(e_{k-n})p_{m}(e_{k})}{p_{J}(e_{k-n})p_{m_{j}}(e_{k})}=0.\]
We deduce from Lemma \ref{lem 0} that there exists $m_J\ge 1$ such that
for any $n,N\ge 1$ 
\[\inf_{k\ge N}\max_{1\le i\le n}\frac{p_{J}(B_w^ie_k)}{p_{m_{J}}(e_{k})}=0.\]
and in particular that for any $n,N\ge 1$, there exists $k\ge N$ such that for any $1\le i\le nd$, we have
\begin{equation}
ndK_n p_{J}(B_w^ie_k)\le p_{m_{J}}(e_{k}).
\label{add 1}
\end{equation}
Moreover, we can suppose that we have $m_J\ge J$, and if there exists $m\ge 1$ such that $\lim_k \frac{p_J(e_k)}{p_m(e_k)}=0$, we can also suppose that $m_J$ satisfies $\lim_k \frac{p_J(e_k)}{p_{m_J}(e_k)}=0$.
We let $C_J=2$ and for any $j\ne J$, we let $m_j\ge 1$ and $C_j>0$ such that 
\begin{equation}
\sup_{n\ge 0}\limsup_{k>n} \frac{p_{j}(e_{k-n})p_{m_{J}}(e_{k})}{p_{J}(e_{k-n})p_{m_j}(e_{k})}<C_j.
\label{add 2}
\end{equation}
Using \eqref{add 1} and \eqref{add 2}, we can then construct an increasing sequence $(s_j)_{j\ge 1}$ such that for any $j\ge 2$
\begin{enumerate}
\item $\displaystyle s_j-n_{j-1}d>s_{j-1};$
\item $\displaystyle \text{for any $1\le i\le n_{j-1}d$,}\quad n_{j-1}dK_{n_{j-1}} p_{J}(B_w^{i}e_{s_j})\le p_{m_{J}}(e_{s_j})$;
\item $\displaystyle \text{for any $l<j$, for any $0\le i\le n_{j-1}d$,}\quad \frac{p_{l}(e_{s_j-i})p_{m_{J}}(e_{s_j})}{p_{J}(e_{s_j-i})p_{m_l}(e_{s_j})}<C_l$;
\end{enumerate}
and as we suppose that either $|c_0|\le 1$ or $\lim_k \frac{p_J(e_k)}{p_{m_J}(e_k)}=0$, we can also construct $(s_j)_{j\ge 1}$ such that 
\begin{enumerate}
\setcounter{enumi}{3}
\item $\displaystyle \text{for any $1\le k<j$,} \quad |c_0|^{n_{k}} p_{J}(e_{s_j})\le p_{m_J}(e_{s_j})$.
\end{enumerate}
So we have for any $l\le k<j$, for any $1\le i\le n_kd$,
\begin{equation}
\label{add 3}
\begin{aligned}
n_kdK_{n_k} p_{l}(B_w^{i}e_{s_j})&\le n_{j-1}dK_{n_{j-1}}\frac{p_{J}(B_w^{i}e_{s_j})}{ p_{m_{J}}(e_{s_j})}\frac{p_{l}(e_{s_j-i})p_{m_{J}}(e_{s_j})}{p_{J}(e_{s_j-i})}\\
&<C_l p_{m_{l}}(e_{s_j})\quad\text{by (2) and (3)} 
\end{aligned}
\end{equation}
and we also have, for any $l\le k<j$, by (3) and (4),
\begin{equation}
\label{add 4}
  \begin{aligned}
|c_0|^{n_k}p_{l}(e_{s_j})=  |c_0|^{n_k}\frac{p_{J}(e_{s_j})}{ p_{m_{J}}(e_{s_j})}\frac{p_{l}(e_{s_j})p_{m_{J}}(e_{s_j})}{p_{J}(e_{s_j})}
<C_l p_{m_{l}}(e_{s_j}). 
  \end{aligned}
\end{equation}
Therefore, for any $l\le k<j$, we have by \eqref{add 3} and \eqref{add 4}
\begin{equation}
\label{add 5}
\begin{aligned}
p_l(P^{n_k}(B_w)(e_{s_j}))&=p_l\Big(\sum_{i=0}^{n_kd}c_i^{(n_k)}B_w^ie_{s_j}\Big)\le \sum_{i=0}^{n_kd}|c_i^{(n_k)}|p_l(B_w^ie_{s_j})\\
&\le |c_0|^{n_k}p_{l}(e_{s_j})+ n_kdK_{n_k}\max_{1\le i\le n_kd}p_{l}(B_w^{i}e_{s_j})\\
&\le 2C_l p_{m_{l}}(e_{s_j}).
\end{aligned}
\end{equation}
We consider the infinite-dimensional closed subspaces $M_k$ defined by
\[M_k=\overline{\text{span}}\{e_{s_j}:j> k\}.\]
As for any $j>k$, we have $s_j-n_{k}d>s_{j-1}$, we deduce form \eqref{add 5} that for any $l\le k$, for any $x\in M_k$, if $X=\lambda^p(A)$, we have
\begin{align*}
p_l(P^{n_k}(B_w)x)^p&=p_l\Big(P^{n_k}(B_w)\Big(\sum_{j=k+1}^\infty x_{s_j} e_{s_j}\Big)\Big)^p
= \sum_{j=k+1}^\infty |x_{s_j}|^p p_l(P^{n_k}(B_w)(e_{s_j}))^p\\
&\le \sum_{j=k+1}^\infty |x_{s_j}|^p (2C_l)^p p_{m_l}(e_{s_j})^p
= (2C_l)^p p_{m_l}(x)^p
\end{align*}
and if $X=c_0(A)$, we have
\begin{align*}
p_l(P^{n_k}(B_w)x)&=p_l\Big(P^{n_k}(B_w)\Big(\sum_{j=k+1}^\infty x_{s_j} e_{s_j}\Big)\Big)
= \max_{j\ge k+1} |x_{s_j}| p_l(P^{n_k}(B_w)(e_{s_j}))\\
&\le \max_{j\ge k+1} |x_{s_j}| 2C_l p_{m_l}(e_{s_j})
= 2C_l p_{m_l}(x).
\end{align*}
We conclude using Theorem \ref{thm M}.
\end{proof}

This theorem, expressed in terms of the matrix $A$, gives the following result:

\begin{theorem}\label{thm poly2}
Let $B_w:X\rightarrow X$ be a weighted shift, where $X=\lambda^p(A)$ with $1\le p<\infty$ or $X=c_0(A)$ and $P$ a non-constant polynomial such that $P(B_w)$ satisfies the Hypercyclicity Criterion. Suppose that there exists $J\ge 1$ such that for any $m\ge 1$, for any $j\ge 1$, there exists $m_j\ge 1$ such that 
\begin{equation*}
\sup_{n\ge 0}\limsup_k \frac{a_{j,k}a_{m,n+k}}{a_{J,k}a_{m_{j},n+k}}<\infty.
\end{equation*}
If there exists $m\ge 1$ such that for any $n\ge 1$
\[\inf_{k\ge 1}\frac{\prod_{\nu=1}^{n}|w_{\nu + k}|a_{J,k}}{a_{m,n+k}}=0\]
and if one of the following two conditions is satisfied:
\begin{enumerate}[\upshape (1)]
\item $P$ has a constant term $|c_0|\le 1$;
\item there exists $m\ge 1$ such that $\lim_k \frac{a_{J,k}}{a_{m,k}}=0$,
\end{enumerate}
then $P(B_w)$ possesses a hypercyclic subspace.
\end{theorem}
\begin{remark}
We do not suppose in Theorems \ref{thm poly} and \ref{thm poly2} that $B_w$ is hypercyclic.
\end{remark}

\begin{cor}
Let $D:H(\mathbb{C})\rightarrow H(\mathbb{C})$ be the differentiation operator.
For every non-constant polynomial $P$, the operator $P(D)$ possesses a hypercyclic subspace.
\end{cor}
\begin{proof}
We deduce this result directly from Theorem \ref{thm poly2}, from \eqref{diff} and from a result of Godefroy and Shapiro (see \cite{God}) which says that every operator $T:H(\mathbb{C})\rightarrow H(\mathbb{C})$, $T\ne \lambda I$, commuting with $D$ is chaotic and thus satisfies the Hypercyclicity Criterion.
\end{proof}

\begin{cor}\label{cor poly3}
Let $B_w:l^p\rightarrow l^p$ be a weighted shift with $1\le p<\infty$.
If for any $n\ge 1$, \[\inf_{k\ge 1}\prod_{\nu=1}^{n}|w_{\nu + k}|=0\] then
$I+B_w$ possesses a hypercyclic subspace.
\end{cor}
\begin{proof}
It is a direct consequence of Theorem \ref{thm poly2} and a result of Le\'on and Montes (see \cite{LeonMontes2}) which says that for every weighted shift $B_w:l^p\rightarrow l^p$, $I+B_w$ satisfies the Hypercyclicity Criterion.
\end{proof}
\begin{remark}
This last result is not completely satisfying if we compare with the result obtained by Le\`on and Montes for complex space $l^2$ (see \cite{LeonMontes}).
\end{remark}

We can also deduce the following results:
 
\begin{cor}
Let $B_w:X\rightarrow X$ be a weighted shift, where $X=\lambda^p(A)$ with $1\le p<\infty$ or $X=c_0(A)$.
Suppose that there exists $J\ge 1$ such that for any $m\ge 1$, for any $j\ge 1$, there exists $m_j\ge 1$ such that for any $n\ge 0$,
\begin{equation}
\lim_k \frac{a_{j,k}a_{m,n+k}}{a_{J,k}a_{m_{j},n+k}}=0.
\label{eq 0}
\end{equation}
If there exists $m\ge 1$ such that we have
\[\sup_{n\ge 1}\inf_{k\ge 1}\frac{\prod_{\nu=1}^{n}|w_{\nu + k}|a_{J,k}}{a_{m,n+k}}<\infty,\]
then for every non-constant polynomial $P$, if $P(B_w)$ satisfies the Hypercyclicity Criterion, $P(B_w)$ possesses a hypercyclic subspace.
\end{cor}
\begin{proof}
Considering $j=J$ in equation \eqref{eq 0}, we first deduce that for any $m\ge 1$, there exists $m'\ge 1$ such that we have
\begin{equation*}
\lim_k \frac{a_{m,k}}{a_{m',k}}=0.
\end{equation*}
Moreover if there exists $m\ge 1$ such that we have
\[\sup_{n\ge 1}\inf_{k\ge 1}\frac{\prod_{\nu=1}^{n}|w_{\nu + k}|a_{J,k}}{a_{m,n+k}}<\infty,\]
then by Proposition~\ref{4.4}, there exists $m\ge 1$ such that for any $n\ge 1$,
\[\sup_{N\ge 1}\inf_{k\ge N}\frac{\prod_{\nu=1}^{n}|w_{\nu + k}|a_{J,k}}{a_{m,n+k}}\le 1.\]
Therefore, we deduce that there exists $m'\ge 1$ such that we have for any $n\ge 1$,
\[\inf_{k\ge 1}\frac{\prod_{\nu=1}^{n}|w_{\nu + k}|a_{J,k}}{a_{m',n+k}}=0.\]
The assumptions of Theorem \ref{thm poly2} are thus satisfied.
\end{proof}

\begin{cor}
Let $B_w:H(\mathbb{C})\rightarrow H(\mathbb{C})$ be a weighted shift. If there exists $m\ge 1$ such that we have
\[\sup_{n\ge 1}\inf_{k\ge 1}\frac{\prod_{\nu=1}^{n}|w_{\nu + k}|}{m^{n+k}}<\infty,\] then for every non-constant polynomial $P$, if $P(B_w)$ satisfies the Hypercyclicity Criterion, $P(B_w)$ possesses also a hypercyclic subspace.
\end{cor}
\begin{remark}
In particular, if $B_w:H(\mathbb{C})\rightarrow H(\mathbb{C})$ possesses a hypercyclic subspace and $P$ is a non-constant polynomial such that $P(B_w)$ satisfies the Hypercyclicity Criterion, then $P(B_w)$ possesses also a hypercyclic subspace.
\end{remark}
\begin{remark}
Condition \eqref{eq 0} is not satisfied for spaces $l^p(v)$ and $c_0(v)$. In fact, this one implies that $X$ is a Schwartz space (see \cite[Theorem 27.10]{Meise}).
A sufficient condition to satisfy condition \eqref{eq 0} is that for any $j,k\ge1$, we have $a_{j,k}\le a_{j,k+1}$, that there exists $J\ge 1$ such that $\lim_k a_{J,k}=\infty$ and that for any $j\ge 1$, there exists $m_j\ge 1$ such that \[\sup_k \frac{a_{j,k}^2}{a_{m_j,k}}<\infty.\]
The space $H(\mathbb{C})$ satisfies this condition and also, for example, the space $s$ of rapidly decreasing sequences.
\end{remark}

\end{document}